\documentclass[11pt,reqno]{amsart}

\usepackage[b5paper,top=1.2in,left=0.9in]{geometry}
\usepackage{verbatim}
\usepackage{amssymb}
\usepackage{amsmath}
\usepackage{graphicx}
\usepackage{appendix}
\usepackage{color}
\usepackage{amsthm}
\usepackage{tikz}
\usetikzlibrary{arrows,positioning,decorations.pathmorphing,  decorations.markings}

\renewcommand{\d}{\mathrm{d}}
\newcommand{\D}{\mathrm{D}}
\newcommand{\e}{\mathrm{e}}

\newtheorem{Thm}{Theorem}[section]
\newtheorem{Lem}[Thm]{Lemma}
\newtheorem{Prop}[Thm]{Proposition}

\newtheorem{Rem}[Thm]{Remark}
\newtheorem{Def}[Thm]{Definition}

\newtheoremstyle{named}{}{}{\itshape}{}{\bfseries}{.}{.5em}{#1 #3}
\theoremstyle{named}

\def\R{\mathbb{R}}

\def\C{\mathbb{C}}
\def\Z{\mathbb{Z}}

\def\H{\mathbb{H}}
\def\bbD{\mathbb{D}}
\def\M{\mathbb{M}}

\def\cL{\mathcal{L}}
\def\cM{\mathcal{M}}
\def\cN{\mathcal{N}}

\def\cT{\mathcal{T}}

\def\a{\alpha}
\def\b{\beta}
\def\c{\gamma}

\def\D{\Delta}

\def\d{\delta}
\def\e{\epsilon}

\def\h{\theta}

\def\l{\lambda}

\def\S{\Sigma}

\def\t{\tau}

\def\Up{\Upsilon}

\def\w{\omega}

\def\bD{\textbf{D}}

\def\bo{\textbf{o}}

\def\=>{\Longrightarrow}

\def\iff{\Longleftrightarrow}

\def\to{\longrightarrow}

\def\o+{\oplus}
\def\bo+{\bigoplus}
\def\x{\times}
\def\<{\langle}
\def\>{\rangle}
\def\({\left(}
\def\){\right)}
\def\oo{\infty}

\def\^{\wedge}
\def\+{\dagger}

\def\inv{^{-1}}
\def\half{\frac{1}{2}}

\def\dd[#1,#2]{\frac{d#1}{d#2}}
\def\del[#1,#2]{\frac{\partial #1}{\partial #2}}
\def\over[#1]{\overline{#1}}
\def\vec[#1]{\overrightarrow{#1}}

\def\tab{\;\;\;\;\;\;}

\newcommand{\til}[1]{\widetilde{#1}}
\newcommand{\what}[1]{\widehat{#1}}
\newcommand{\xto}[1]{\xrightarrow{#1}}

\newcommand{\veca}[2][cccccccccccccccccccccccccccccccccccccccccc]{\left(\begin{array}{#1}#2 \\ \end{array} \right)}

\newcommand{\Eq}[1]{\begin{align}#1\end{align}}
\newcommand{\Eqn}[1]{\begin{align*}#1\end{align*}}

\tikzset{->-/.style={decoration={
  markings,
  mark=at position .5 with {\arrow{latex}}},postaction={decorate}}}
\tikzset{
    >=latex
    }
\begin{document}
\title{On Ramond decorations}

\author{  Ivan C.H. Ip}

\address[Ivan C.H. Ip]{
         	   Center for the Promotion of
         	  Interdisciplinary Education and Research/\newline
     	  Department of Mathematics, Graduate School of Science,  Kyoto University, Japan
		\newline
		Email: ivan.ip@math.kyoto-u.ac.jp
		\newline URL: http://math.kyoto-u.ac.jp/\textasciitilde ivan.ip/
            }

\author{ Robert C. Penner}
\address[Robert C. Penner]{
Institute des Hautes \'Etudes Scientifiques, Bures-sur-Yvette, France; University of California--Los Angeles, Los Angeles, USA 
\newline 
Email: rpenner@ihes.fr}  
\author{Anton M. Zeitlin}
\address[Anton M. Zeitlin]{
          Department of Mathematics, Louisiana State University, Baton Rouge, USA; IPME RAS, St. Petersburg
          \newline
          Email: anton.zeitlin@gmail.com
          \newline URL: http://math.lsu.edu/\textasciitilde zeitlin/}

\date{\today}

\numberwithin{equation}{section}

\maketitle

\begin{abstract}
We impose constraints on the odd coordinates of super Teichm\"uller space in the uniformization picture for the monodromies around Ramond punctures, thus reducing the overall odd dimension to be compatible with that of the moduli spaces of super Riemann surfaces. Namely, the monodromy of
a puncture must be a true parabolic element of the canonical
subgroup $PSL(2,\mathbb{R})$.
\end{abstract}

\section{Introduction: super Riemann surfaces, punctures and uniformization}\label{sec:intro}

Moduli spaces of complex supermanifolds of dimension $(1|1)$ and the subclass of those called super Riemann surfaces,  are the cornerstones of superstring theory (for a review see \cite{donagi,witten} as well as the original papers \cite{CR, schwarz}). 

In our previous work \cite{PZ,IPZ}, we developed a coordinate system for super Teichm\"uller space associated to super Riemann surfaces with punctures, based upon uniformization taking Poincar\'e metrics. These are generalizations of standard Penner coordinates \cite{penner,pbook} on the Teichm\"uller space. Each such super Riemann surface is defined in this uniformization approach as a factor $\what{\H}^+/\Gamma$ of a super analogue $\what{\H}^+$ of the upper half-plane $\H^+$, modulo the action of a discrete Fuchsian subgroup $\Gamma$ of $OSp(1|2)$ acting on $\what{\H}^+$ as a superconformal transformation, which is a certain generalization of a standard fractional-linear transformation, see e.g. \cite{Na}. 

The super Teichm\"uller space $S\mathcal{T}(F)$, where $F$ is the underlying Riemann surface of genus $g$ with $s$ punctures, is defined in full analogy with standard pure even case, viewed as a character variety:  
\Eq{S\mathcal{T}(F)=\mathrm{Hom}'(\pi_1(F)\to OSp(1|2))/OSp(1|2),} so that $\Gamma$ above belongs to the image. Here $\pi_1(F)$ is the fundamental group of the underlying Riemann surface with punctures, and $\mathrm{Hom}'$ in \cite{PZ}, \cite{IPZ} stands for the homomorphisms which map elements of $\pi_1(F)$ corresponding to small loops around the punctures, to parabolic elements of $OSp(1|2)$, which means that their natural projections to $PSL(2,\R)$ are parabolic elements. 

It turns out that the dimension of the resulting space is $(6g-6+2s|4g-4+2s)$. 
However, it is known from the study of moduli space $\cM(\S)$ of a super Riemann surface $\S$ that its dimension is a little more subtle, namely it is 
\Eq{\dim \cM(\S)=(6g-6+2s|4g-4+2n_{NS}+n_R),} where $n_{NS}$ and $n_{R}$ are known as the number of Neveu-Schwarz (NS) and Ramond punctures on $\S$ respectively. These punctures are the analogs of punctures or marked points on an ordinary Riemann surface, in which on super Riemann surface they are described by codimension $(1|0)$ divisors.

Let us look at these classes of punctures in detail (we refer to \cite{witten} for more information) from a point of view of complex supermanifolds and compare it to the uniformization approach to understand where this discrepancy in dimension count comes from. 

We recall (see e.g. \cite{witten}) that a \emph{super Riemann surface} $\S$ is a $(1|1)$-dimensional complex supermanifold, i.e. locally isomorphic to $\C^{1|1}$, together with a subbundle $\mathcal{D}\subset T\S$ of rank $(0|1)$, such that if $D$ is its nonzero section (in some open set $U\subset \S$), $D^2=\half\{D,D\}$ is nowhere proportional to $D$. The local coordinates $(z, \theta)\in\C^{1|1}$, where $z$ is even and $\h$ is odd, such that \Eq{
D=D_\h:=\partial_{\theta}+\theta\partial_z,} are known as \emph{superconformal coordinates} and can be chosen for any such nonzero section $D$. The \emph{superconformal transformations} operating between patches are the ones preserving $\mathcal{D}$. 

The \emph{NS puncture} is a natural generalization of the puncture of ordinary Riemann surface, and can be considered as any point $(z_0,\theta_0)$ on the super Riemann surface. Locally one can associate to it a $(0|1)$-dimensional divisor of the form $z=z_0-\theta_0\theta$, which is the orbit with respect to the action of the group generated by $D$, and this divisor uniquely determine the point $(z_0,\theta_0)$ due to the superconformal structure.

Let us consider the case when the puncture is at $(0,0)$ locally. In its neighborhood let us pick a coordinate transformation 
\Eq{
z=e^{w}, \tab \theta= e^{w/2}\eta, \label{trans1}} such that the neighborhood (without the puncture) is mapped to a \emph{supertube} with $w$ sitting on a cylinder $w \sim w+2\pi i$, and $D_\h$ becomes
\Eq{
D_{\theta}=e^{-w/2}(\partial_{\eta}+\eta\partial_{w}).
} Hence $(w,\eta)$ are superconformal coordinates, and we have the full equivalence relation given by \Eq{w \sim w+2\pi i, \tab \eta\to -\eta.\label{eq1}}

Therefore, in the uniformization picture, the group element, corresponding to the monodromy around NS puncture should be conjugate to an element of $OSp(1|2)$ corresponding to fractional linear transformation representing translation and an odd element reflection. In other words, an element of a Borel subgroup of $OSp(1|2)$ generated by the maximal negative root (which is an $SL(2,\R)$ generator), accompanied by a fermionic reflection.  

The case of \emph{Ramond puncture} is a whole different story. On the level of super Riemann surfaces, the associated divisor is determined as follows. In this case, we are looking at the 
case when the condition that  $D^2$ is linearly independent of $D$ is violated along some $(0|1)$ divisor. Namely, in some local coordinates $(z,\theta)$ near the Ramond puncture with coordinates $(0,0)$, $\mathcal{D}$ has a section of 
the form 
\Eq{D=\partial_\theta+z\theta \partial_z.}
We see that its square vanishes along the $Ramond$ $divisor$ $z=0$. One can map the neighborhood patch to the supertube using a different coordinate transformation 
\Eq{z=e^{w},\tab \theta=\eta\label{trans2},} those coordinates on the supertube will be superconformal, since \Eq{
D=\partial_{\eta}+\eta\partial_{w}.} Notice that the identifications we have to impose on $(w, \eta)$ now become: \Eq{w \sim w+2\pi i,\tab \eta\to+\eta.\label{eq2}} Again, we see that the group element corresponding to monodromy around the loop should belong to the same subgroup in the Borel subgroup as in the NS case, just without extra reflection.

What lesson did we learn from this discussion? Previously, in \cite{PZ,IPZ} we obtained a bigger Teichm\"uller space than needed for the study of supermoduli, since  the constraint we imposed on $\Gamma$ with respect to the monodromies around punctures was too weak. The condition which is needed to reduce dimension appropriately has to be the following: {\it group elements, corresponding to the monodromies around punctures, should lie in the conjugacy classes of parabolic elements of the canonical $SL(2,\R)$ subgroup of $OSp(1|2)$}. This constraint will remove the necessary $(0|n_{R})$-dimensional bundle from $S\mathcal{T}(F)$, which we will call \emph{Ramond decoration}.

The main goal of this note is to express this constraint in terms of coordinates obtained in \cite{PZ,IPZ} and we will see that it is indeed an elegant formula, linear in odd variables. 

The structure of the paper is as follows. In Section \ref{sec:coord} we review construction of \cite{PZ,IPZ}.  In Section \ref{sec:frac}, we recall the representation of $OSp(1|2)$ as fractional linear transformation. Section \ref{sec:mono} is devoted to the explicit study of the monodromy around the punctures and, finally in Section \ref{sec:dim} we derive the formula for the aforementioned constraint.

\section{Coordinates of super Teichm\"uller space}\label{sec:coord}
Let us first recall the ingredients used to construct the super Teichm\"uller space for the case $\cN=1$. We will adapt the notation and construction from \cite{IPZ} for the case $\cN=2$ which restricts to the case $\cN=1$.

\subsection{Definition of $OSp(1|2)$}
Let the supergroup $SL(1|2)$ be $(2|1)\x (2|1)$ supermatrices with superdeterminant equal to 1, where we write
\Eq{g=\veca{a&\a&b\\\c&f&\b\\c&\d&d}\in SL(1|2)}
such that $a,b,c,d,f$ are even entries, and $\a,\b,\c,\d$ are odd entries, with the supernumber defined over $\R$. We will consider the component $SL(1|2)_0$ where $f>0$. The \emph{superdeterminant} or \emph{Berezinian} is defined to be
\Eq{sdet(g):=f\inv \det\left(\veca{a&b\\c&d}+f\inv \veca{\a\c&\a\d\\ \b\c&\b\d}
\right),}
while the \emph{supertrace} is given by 
\Eq{str(g):=a+d-f.}
Let us denote by 
\Eq{
J:=\veca{0&0&1\\0&1&0\\-1&0&0}}
with $sdet(J)=1$, and define the \emph{supertranspose} as
\Eq{
g^{st}:=\veca{a&\c&c\\-\a&f&-\d\\b&\b&d}.}
Then the supergroup $OSp(1|2)\subset SL(1|2)_0$ is defined to be the supermatrices $g\in SL(1|2)_0$ satisfying
\Eq{
g^{st}J g = J.}
We have a natural projection $SL(1|2)_0\to SL(2,\R)$ given by
\Eq{g=\veca{a&\a&b\\\c&f&\b\\c&\d&d}\mapsto \frac{1}{\sqrt{f_\#}}\veca{a_\#&b_\#\\c_\#&d_\#},}
where $a_\#$ denotes the {\it body} of the supernumber $a$.

Finally we denote two special types of matrices in $OSp(1|2)$ by
\Eq{
D_a:=\veca{a&0&0\\0&1&0\\0&0&a\inv}, \tab Z_a:=\veca{a&0&0\\0&a^2&0\\0&0&a}
}
that will be useful later on. The matrix $Z_{-1}$ is also known as the \emph{fermionic reflection}.
\subsection{Decorated super Teichm\"uller space}
Let $F:=F_g^s$ be a Riemann surface with genus $g\geq 0$ and $s\geq 1$ punctures such that $2g+s-2>0$. Let $\D$ be an ideal triangulation of $F$ whose vertices are lifted to the set of vertices $\til{\D}_\oo$ at infinity on the universal cover $\til{F}\simeq \bbD$ where $\bbD$ is the Poincar\'e unit disk, and $\pi_1(F)$ acts on $\bbD$ by the Deck transformations.

The coordinates of the decorated super Teichm\"uller space $S\til{\cT}(F)$ can be realized in Minkowski space $\M=\R^{2,1|2}$ with inner product between two vectors $A=(x_1,x_2,y|\phi,\h),A'=(x_1',x_2',y'|\phi',\h')$ in $\M$ given by
\Eq{
\<A,A'\>:=\frac12(x_1x_2'+x_1'x_2)-yy'+\phi\h'+\phi'\h,}
where the square-root of such inner product is called a \emph{$\l$-length}. We define the \emph{positive light cone} to be
\Eq{
\cL:=\{ A\in \M: \<A,A\>=0,  x_1>0, x_2>0\}.}
Any point in $\M$ can also be represented as 
\Eq{
M_c=\veca{x_1&\phi&y-c\\-\phi&c&\-\h\\ y+c&\h&x_2}\in \M,}
where $\cL$ corresponds to the subspace with $c=0$. Then $OSp(1|2)$ acts naturally on $\M$ and $\cL$ by the adjoint action
\Eq{g\cdot M_c:= g^{st} M_c g,\tab g\in OSp(1|2),} and every vectors in the light cone $\cL$ can be put into the form ${e_\h:=(1,0,0|0,\pm\h)\in \cL}$ for some odd parameter $\h$. The $OSp(1|2)$-orbit of
$$e_0:=(1,0,0|0,0)\in \cL,$$
is of special importance, and we will refer to
\Eq{\cL_0:=OSp(1|2)\cdot e_0\subset \cL}
as the \emph{special light cone}.

\begin{Lem}\cite{PZ} Let $\D ABC$ be a \emph{positive triple} (i.e. the bosonic part of $(A,B,C)$ are positively oriented) in the special light cone $\cL_0$. Then there is a unique $g\in OSp(1|2)$ (up to composition by the fermionic reflection $Z_{-1}$), unique even $r,s,t>0$, and odd $\phi$ such that
\Eq{g\cdot A=r(0,1,0|0,0),\tab g\cdot B=t(1,1,1|\phi,\phi),\tab g\cdot C=s(1,0,0|0,0).
\label{standpos}}
\end{Lem}
Coordinates of the form \eqref{standpos} are said to be in \emph{standard position}. We then have two associated transformations:
\begin{Lem}The \emph{prime transformations} $P_\h^\pm$ are given by
\Eq{P_\h^+:=\veca{-1&\h&1\\-\h&1&0\\-1&0&0},\tab P_\h^-:=(P_\h^+)\inv =\veca{0&0&-1\\0&1&-\h\\1&-\h&-1},}
which rotates the standard position with odd parameters $\h$, from $\D ABC$ to $\D BCA$ and $\D CAB$ respectively.

The \emph{upside-down} transformations $\Up^\chi$ are given by
\Eq{\Up^{\chi}:=J\circ D_{\sqrt{\chi}},
}
which sends the standard position from $\D ABC$ to $\D CDA$ in a quadrilateral $\diamondsuit ABCD$ with \emph{cross ratio}
\Eq{
\chi:=\frac{ac}{bd},}
where $a^2=\<A,B\>, b^2=\<B,C\>, c^2=\<C,D\>$ and $d^2=\<D,A\>$.
\end{Lem}
One readily checks that $P^{\h,\pm}$ and $\Up^\chi$ are all elements in $OSp(1|2)$.
\begin{Def} The \emph{decorated super Teichm\"uller space} $S\til{\cT}(F)$ is the space of $OSp(1|2)$-orbits of lifts 
\Eq{\ell:\til{\D}_\oo\to\cL_0,} which are $\pi_1$-equivariant for some super Fuchsian representation $\what{\rho}:\pi_1\to OSp(1|2)$. More precisely,
\begin{itemize}
\item[(1)] $\what{\rho}(\c)(\ell(a)) = \ell(\c(a))$ for each $\c\in \pi_1$ and $a\in \til{\D}_\oo$;
\item[(2)] the natural projection
\Eq{
\rho:\pi_1\xto{\what{\rho}} OSp(1|2)\to SL(2,\R)\to PSL(2,\R)}
is a Fuchsian representation.
\end{itemize}
\end{Def}
In \cite{PZ}, the space of all such lifts are constructed by a recursive procedure to lift the ideal triangles of $\til{F}$ on the universal cover to the light cone, using the ``basic calculations" to determine the $OSp(1|2)$-orbits on the light cone. 
The construction is subsequently simplified and generalized to $\cN=2$ in \cite{IPZ}, where  the construction of the lift was directly connected to the combinatorial description of spin structures, discovered in \cite{PZ}. There, the spin structures were identified  with classes of orientations of the trivalent fatgraph spine $\tau$, dual 
to the triangulation $\Delta$: two orientations belong to the same equivalence class if they are related by sequence of reversals of orientation of all edges incident to a given vertex. It was shown in \cite{PZ}, that under the elementary Whitehead move (flip), the orientations change as in Figure \ref{flipgraphint} in the generic situation. 
\begin{figure}[h!]
\begin{center}
\begin{tikzpicture}[ ultra thick, baseline=1cm]
\draw (0,0)--(210:1) node[above] at (210:0.7){$\epsilon_2$};
\draw (0,0)--(330:1) node[above] at (330:0.7){$\epsilon_4$};
    \draw[ 
 	ultra thick,
        decoration={markings, mark=at position 0.5 with {\arrow{>}}},
        postaction={decorate}
        ]
        (0,0) -- (0,2);
\draw[yshift=2cm] (0,0)--(30:1) node[below] at (30:0.7){$\epsilon_3$};
\draw[yshift=2cm] (0,0)--(150:1) node[below] at (150:0.7){$\epsilon_1$};
\end{tikzpicture}
\begin{tikzpicture}[baseline]
\draw[->, thick](0,0)--(2,0);
\node[above] at (1,0) {};
\node at (-1,0){};
\node at (3,0){};
\end{tikzpicture}
\begin{tikzpicture}[ultra thick, baseline]
\draw (0,0)--(120:1) node[above] at (100:0.3){$\epsilon_1$};
\draw (0,0)--(240:1) node[below] at (260:0.3){$\epsilon_2$};
    \draw[ 
        decoration={markings, mark=at position 0.5 with {\arrow{<}}},
        postaction={decorate}
        ]
        (0,0) -- (2,0);
\draw[xshift=2cm] (0,0)--(60:1) node[above] at (100:0.3){$-\epsilon_3$};
\draw [xshift=2cm](0,0)--(-60:1) node[below] at (-80:0.3){$\epsilon_4$};
\end{tikzpicture}
\end{center}
\caption{Spin graph evolution in the generic situation}
\label{flipgraphint}
\end{figure}
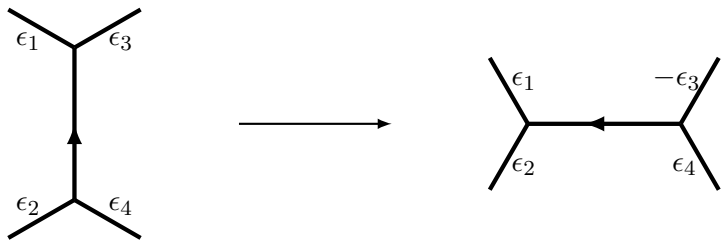
There $\epsilon_i$ stands for the orientation of the corresponding edge and negative sign represents orientation reversal.

It was also explicitly shown that this description of  spin structures is compatible with Natazon's description \cite{Na} of a choice of lift $\til{\rho}:\pi_1\to OSp(1|2)\to SL(2,\R)$ of the Fuchsian representation $\rho$.

Altogether this leads to the following description of the coordinates on $S\til{\cT}(F)$:
\begin{Thm}
i) The components of $S\til{\cT}(F)$ is determined by the space of spin structures on $F$. To each component $C$ of $S\til{\cT}(F)$, there are global affine coordinates on $C$ given by assigning to a triangulation $\D$ of $F$, 
\begin{itemize}
\item one even coordinate called $\l$-length for each edge; 
\item one odd coordinate called $\mu$-invariant for each triangle, taken modulo an overall change of sign.
\end{itemize} In particular we have a real-analytic homeomorphism
\Eq{C\to \R_{>0}^{6g-6+3s|4g-4+2s}/\Z_2.}
ii) The super Ptolemy transformations \cite{PZ} provide the analytic relations between coordinates assigned to different choice of triangulation $\D'$ of $F$, namely upon flip transformation. Explicitly (see Figure \ref{ptolemy}),  when all $a,b,c,d$ are different edges of the triangulations of $F$, Ptolemy transformations are as follows:

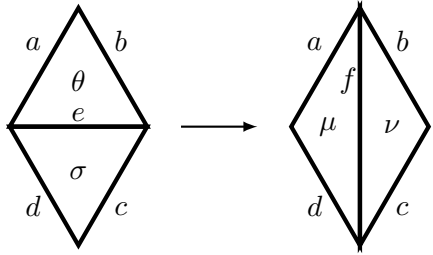
\begin{figure}[h!]

\centering

\begin{tikzpicture}[scale=0.6, baseline,ultra thick]

\draw (0,0)--(3,0)--(60:3)--cycle;

\draw (0,0)--(3,0)--(-60:3)--cycle;

\draw node[above] at (70:1.5){$a$};

\draw node[above] at (30:2.8){$b$};

\draw node[below] at (-30:2.8){$c$};

\draw node[below=-0.1] at (-70:1.5){$d$};

\draw node[above] at (1.5,-0.15){$e$};

\draw node[left] at (0,0) {};

\draw node[above] at (60:3) {};

\draw node[right] at (3,0) {};

\draw node[below] at (-60:3) {};

\draw node at (1.5,1){$\theta$};

\draw node at (1.5,-1){$\sigma$};

\end{tikzpicture}
\begin{tikzpicture}[baseline]

\draw[->, thick](0,0)--(1,0);

\node[above]  at (0.5,0) {};

\end{tikzpicture}
\begin{tikzpicture}[scale=0.6, baseline,ultra thick]

\draw (0,0)--(60:3)--(-60:3)--cycle;

\draw (3,0)--(60:3)--(-60:3)--cycle;

\draw node[above] at (70:1.5){$a$};

\draw node[above] at (30:2.8){$b$};

\draw node[below] at (-30:2.8){$c$};

\draw node[below=-0.1] at (-70:1.5){$d$};

\draw node[left] at (1.7,1){$f$};

\draw node[left] at (0,0) {};

\draw node[above] at (60:3) {};

\draw node[right] at (3,0) {};

\draw node[below] at (-60:3) {};

\draw node at (0.8,0){$\mu$};

\draw node at (2.2,0){$\nu$};

\end{tikzpicture}\\
\caption{Generic flip transformation}
\label{ptolemy}
\end{figure}


\begin{eqnarray}
&& ef=(ac+bd)\Big(1+\frac{\sigma\theta\sqrt{\chi}}{1+\chi}\Big),\nonumber\\
&&\nu=\frac{\sigma+\theta\sqrt{\chi}}{\sqrt{1+\chi}},\quad
\mu=\frac{\sigma\sqrt{\chi}-\theta}{\sqrt{1+\chi}},
\end{eqnarray}
where $\chi=\frac{ac}{bd}$,  so that the evolution of arrows is as in  Figure \ref{flipgraphint}.
\end{Thm}

The decorated super Teichm\"uller space $S\til{\cT}(F)$ is naturally a principal $\R_+^s$-bundle over the super Teichm\"uller space $S\cT(F)$ defined in the introduction, and one can descend to $S\cT(F)$ by taking the appropriate shear coordinates around punctures as in the bosonic case \cite{penner}. 
\subsection{Construction of the lift}
In this section we recall the construction of the lift $\what{\rho}:\pi_1(F)\to OSp(1|2)$ given by the coordinates of the decorated Teichm\"uller space described in \cite{PZ, IPZ}. Let us fix a spin structure $\w$ corresponding to the component of this lift, which is represented by an orientation of the fatgraph spine $\t$ of the triangulation. 

The fundamental domain $\bD\subset\til{F}$ on the universal cover of $F$ is naturally a $(4g+2s)$-gon. It suffices to determine the image of the generators $\c_i\in \pi_1(F)$, $i=1,...,2g+s$, which identifies a pair of frontier edges $c_i,c_i'$ of $\bD$. Let $c_i'=\c_i(c_i)$. To determine the image of $\what{\rho}(\c_i)\in OSp(1|2)$, let $\D ABC$ and $\D A'B'C'$ be the lift of the unique pair of triangles such that $BC=\ell(c_i), B'C'=\ell(c_i'), \ell\inv (\D ABC)\subset \bD$, and $\ell\inv(\D A'B'C')\not\subset \bD$. Then by definition of $\ell$ there is a unique transformation $g\in OSp(1|2)$ bringing the standard position from $\D ABC$ to $\D A'B'C'$ and matching $BC$ to $B'C'$. Explicitly let $\c_i$ be homotopically represented by a path in $\t$. Then 
\begin{Prop}
The image $\what{\rho}(\c_i):=g\in OSp(1|2)$ is a composition of the form
\Eq{
\what{\rho}(\c_i)=\prod_k Z_{\e_k}\circ \Up^{\chi_k}\circ P_{\h_k}^{\pm}\in OSp(1|2),}
where
\begin{itemize}
\item $\e_k\in\{1,-1\}$ according to whether the segment of $\c_i$ is aligned with the orientation of $w$.
\item $\Up^{\chi_k}$ is the upside-down transformation corresponding to the pair of triangles crossing the segment of $\c_i$.
\item $P_{\h_k}^{\pm}$ is the prime transformation corresponding to $\c_i$ turning left ($+$) or right ($-$) at the vertex of $\t$.
\end{itemize}
\end{Prop}
\section{$OSp(1|2)$ as fractional linear transform}\label{sec:frac}
In this section, we recall the representation of $OSp(1|2)$ as fractional linear transformation on the super upper half-plane.

Recall that in the bosonic case, $PSL(2,\R)$ acts transitively on the upper half-plane $\H^+:=\{x+iy|y>0\}\subset \C$ by
\Eq{z\mapsto \frac{az+b}{cz+d},}
where $z=x+iy$ and $\veca{a&b\\c&d}\in PSL(2,\R)$. 

In the super case, we have an analogue given as follows. Let $\C^{1|1}$ be the complex superplane, and consider the super upper half-plane $\what{\H}^+:=\{(z,\eta)|Im(z_\#)>0\}\subset \C^{1|1}$ where $z_\#$ denote the body of $z$. Then it is well-known that
\begin{Prop}\cite{CR, PZ, witten} An element $g=\veca{a&\a&b\\\c&f&\b\\c&\d&d}\in OSp(1|2)$ acts transitively on $\what{\H}^+$ by the superconformal transformations
\Eq{
z&\mapsto \frac{az+b}{cz+d}+\eta\frac{\c z+\d}{(cz+d)^2},\\
\eta&\mapsto \frac{\c z+\d}{cz+d}+\eta\frac{1+\half \d\c}{cz+d}.
}
\end{Prop}
In particular, a translation of the form $z\mapsto z+b$ is given by an element of a Borel subgroup of the form $g=\veca{\pm1&0&b\\0&1&0\\0&0&\pm1}$ which belongs to the canonical  $SL(2,\R)$ subgroup of $OSp(1|2)$.

Return to our setting, recall from the introduction that we considered the conformal transformations \eqref{trans1}, \eqref{trans2} from the neighborhood around the punctures to the supertube. If we restrict to the neighborhood within $0<|z|<1$, then the image of the transformation becomes the left half-plane $\{(z,\eta)|Re(z_\#)<0\}\subset\C^{1|1}$ instead. Therefore rotating our setting by 90 degree, the equivalence relations given by the simple translation of the even variable \eqref{eq1}, \eqref{eq2} in the complex direction are then represented by the action of the lower Borel elements
\Eq{\veca{\pm1&0&0\\0&1&0\\b&0&\pm1}\in OSp(1|2).}

\section{Monodromies and decorations}\label{sec:mono}
We are now ready to discuss the monodromies around punctures. The construction of the map $\what{\rho}:\pi_1(F)\to OSp(1|2)$ requires that the group elements $\what{\rho}(\c)$ corresponding to loops $\c$ going around punctures, which being projected to $PSL(2,\R)$ is parabolic, i.e. is conjugate to $\pm\veca{1&0\\b&1}$.

For a loop $\c\in \pi_1(F)$ around a puncture, by definition the monodromy $\what{\rho}$ should fix the lift of the puncture $A\in \cL_0$ which belongs to a sequence of triangles $\D AB_{i+1}B_i$ (see Figure \ref{punc}). Acting by $OSp(1|2)$ if necessary, let us choose the standard position such that the puncture is lifted to $A=r(0,1,0|0,0)$ for some $r>0$.
\begin{figure}[htb!]
\centering
\begin{tikzpicture}[baseline=(0), every node/.style={inner sep=0, minimum size=0.2cm, circle, fill=black }, x=0.7cm, y=0.7cm]
\node[red] (0) at (0,0){};
\node [red,fill=none] at (0,-0.5){$A$};
\node[gray] (1) at (0:4){};
\node[gray] (2) at (60:4){};
\node[gray] (3) at (120:4){};
\node[gray] (4) at (180:4){};
\node[gray] (5) at (240:4){};
\node[gray] (6) at (300:4){};
\node [gray, fill=none] at (0:4.5){$B_1$};
\node [gray, fill=none] at (60:4.5){$B_2$};
\node [gray, fill=none] at (120:4.5){$B_3$};
\node [gray, fill=none] at (180:4.5){$B_4$};
\node [gray, fill=none] at (240:4.5){$B_5$};
\node [gray, fill=none] at (300:4.5){$B_n$};
\node [gray, fill=none] at (0,-4){\huge $\cdots$};
\draw[gray] (1)--(2)--(3)--(4)--(5)--(6)--(1);
\draw[gray] (1)--(0)--(2);
\draw[gray] (3)--(0)--(4);
\draw[gray] (5)--(0)--(6);
\draw[red,->-, very thick] (30:2.3)--(90:2.3);
\draw[red,->-, very thick] (90:2.3)--(150:2.3);
\draw[red,->-, very thick] (150:2.3)--(210:2.3);
\draw[red,->-, very thick] (210:2.3)--(270:2.3);
\draw[red,->-, very thick] (270:2.3)--(330:2.3);
\draw[red,->-, very thick] (330:2.3)--(30:2.3);
\draw (30:2.3)--(30:4.3);
\draw (90:2.3)--(90:4.3);
\draw (150:2.3)--(150:4.3);
\draw (210:2.3)--(210:4.3);
\draw (270:2.3)--(270:4.3);
\draw (330:2.3)--(330:4.3);
\node [rotate=30, red, fill=none] at (300:2.3){\huge$\cdots$};
\node [red, fill=none] at (2.5,-0.5){$\c$};
\node [red, fill=none, below right] at (30:2.3){$\t_1$};
\node [red,fill=none, below] at (90:2.3){$\t_2$};
\node [red,fill=none, below right] at (150:2.3){$\t_3$};
\node [red,fill=none, below] at (210:2.3){$\t_4$};
\node [red,fill=none, below right] at (270:2.3){$\t_5$};
\node [red,fill=none, below] at (330:2.3){$\t_n$};
\end{tikzpicture}
\caption{The loop $\c$ in $\t$ around a puncture $A$ surrounded by $n$ triangles}
\label{punc}
\end{figure}
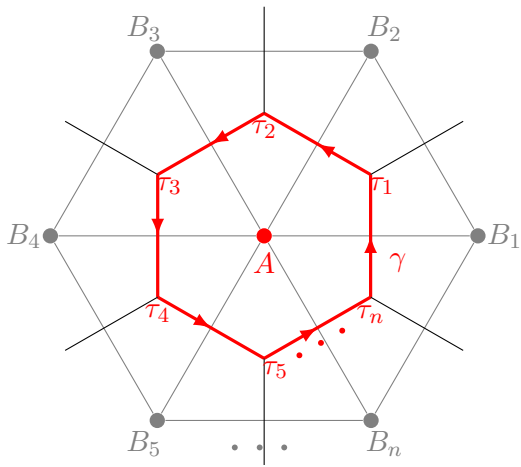

Then by definition of the adjoint action of $OSp(1|2)$ on $\cL_0$, a generic element $\what{\rho}(\c)$ of the monodromy is given by a lower triangular matrix
\Eq{
g=\veca{\pm1&0&0\\\h&1&0\\b&\phi&\pm1}\in OSp(1|2),\label{gg}}
and the conditions for it to be in $OSp(1|2)$ require that $\phi=\mp \h$, hence we only have one odd parameter $\h$.

Let us call the punctures corresponding to the monodromy whose projection to $SL(2,\R)$ has trace $2$ (resp. $-2$) as Ramond (resp. Neveu-Schwarz (NS)) punctures. 

In the introduction we pointed out that in order to obtain the punctured super Riemann surfaces with proper boundary conditions on Ramond and NS punctures, we need to impose the condition that the monodromy element around the puncture be conjugate to a simple translation of the even variable in the transformed space, in which $OSp(1|2)$ acts by fractional-linear transformation. We saw in the previous section that this corresponds to an element of the lower Borel subgroup of the $SL(2,\R)$ subgroup of $OSp(1|2)$.

Hence we now tighten things a bit and require also that the element $g$ in \eqref{gg} above be conjugate to 
\Eq{\veca{\pm1&0&0\\0&1&0\\b&0&\pm1}\in OSp(1|2)\label{standform}}
by elements of $OSp(1|2)$ which fix the puncture. In other words, 
\begin{Def}[Monodromy constraint]
For every lift of puncture $p\in \cL_0$, the $OSp(1|2)$-orbit of pairs $(p,g)\in \cL_0\x OSp(1|2)$ under the action given by
\Eq{
U\cdot (p,g) = (U\cdot p, UgU\inv),\tab U\in OSp(1|2),
}
is required to contain a point $(p,g_0)$ where $g_0$ is of the form \eqref{standform}.
\end{Def}
Let us see how these constraints affect each type of punctures.

\begin{Lem} In the case of NS puncture, the monodromy constraint is always satisfied.
\end{Lem}
\begin{proof} Take $U:=\veca{1&0&0\\\h/2&1&0\\0&-\h/2&1}$, then it fixes $p=A$. We have  $$U\inv=\veca{1&0&0\\-\h/2&1&0\\0&\h/2&1}$$ and
$$U\cdot\veca{-1&0&0\\\h&1&0\\b&-\h&-1}\cdot U\inv =\veca{-1&0&0\\0&1&0\\b&0&-1}$$
as required.
\end{proof}
However, for Ramond puncture, the situation is different.
\begin{Lem}
The monodromy constraint is satisfied if and only if $\h=0$.
\end{Lem}
\begin{proof}
The conjugation by a general $U$ can be rewritten as
$$\veca{a&\a&b\\\c&f&\b\\c&\d&d}\veca{1&0&0\\\h&1&0\\B&-\h&1}=\veca{1&0&0\\0&1&0\\B'&0&1}\veca{a&\a&b\\\c&f&\b\\c&\d&d}$$
which leads to the constraints
\Eqn{
Bb=\a\h, \tab b\h=\b\h=0, \tab f\h+B\b=0, \tab B'a+\h\d=Bd, \tab B'\a+d\h=0.}
Since we also need to fix the vector corresponding to the puncture $p=A$, the conjugation requires $U$ to be lower triangular. In particular $\b=0$, hence $f\h=0\=> \h=0$ since $f>0$.
\end{proof}
Therefore to get a ``true" super Teichm\"uller space, we would like to impose the conditions $\h=0$ to our group elements corresponding to monodromies around Ramond punctures.
\section{Dimension reduction}\label{sec:dim}
In this section, we show that the conditions that $\h=0$ for monodromies around Ramond punctures will impose a linear constraint in terms of the odd coordinates of $S\til{\cT}(F)$. In particular, this will reduce the dimension of $S\til{\cT}(F)$, and hence $S\cT(F)$, making the overall odd dimension $4g-4+2n_{NS}+n_R$, where $n_{NS}$ and $n_R$ are the number of NS and Ramond punctures respectively, as explained in the introduction.

Fix a fatgraph orientation $\w$ on the fatgraph spine $\t$ corresponding to a spin structure. Consider a loop homotopic to $\c=(\t_1,\t_2,...,\t_n,\t_1)\in \pi_1(F)$ on $\t$ passing through the vertices $\t_i\in\t$ associated to the triangles around a Ramond puncture, and we assume $\c$ is going in a counter-clockwise direction (cf. Figure \ref{punc}). According to our construction, the monodromy is given by
\Eq{
\what{\rho}(\c)=\prod_{k=1}^n Z_{\e_k}\circ \Up^{\chi_k}\circ P_{\h_k}^+\in OSp(1|2),}
where the product is read from right to left, $\h_i$ is the odd parameter for the triangle $\t_i$, $\chi_i$ is the cross ratio between $\t_i$ and $\t_{i+1}$, and $\e_i=-1$ if the path $\c$ at $\t_i\to \t_{i+1}$ has the same orientation as $\w$ or else $\e_i=+1$ otherwise.
\begin{Prop} $\what{\rho}(\c)\in OSp(1|2)$ is of the form
\Eq{
\what{\rho}=\veca{*_0&0&0\\-*_1&1&0\\-*_3&*_2&*_0},}
where $c_k:=-\e_k\sqrt{\chi_k}$, and
\Eq{
*_0&:=\prod_{k=1}^nc_k=(-1)^n\prod_{k=1}^n \e_k,\\
*_1&:=\sum_{k=1}^{n} \h_k\prod_{j=1}^{k-1}c_j\inv,\\
*_2&:=\sum_{k=1}^n \h_k \prod_{j=k}^n c_j = (*_0)(*_1),\\
*_3&:=\sum_{j=1}^n\left(\prod_{k=1}^{j-1}c_k\inv\prod_{k=j}^nc_k\right)+\sum_{1\leq i<j\leq n} \left(\h_i\h_j \prod_{k=1}^{j-1}c_k\inv\prod_{k=j}^n c_k\right).}
\end{Prop}
\begin{proof}Writing out the matrices explicitly, we see that
$$Z_{\e_k}\circ \Up^\chi\circ P_{\h_k}^+=\veca{-\e_k\sqrt{\chi_k}\inv&0&0\\-\h_k&1&0\\\e_k\sqrt{\chi_k}&-\h_k \e_k\sqrt{\chi_k}&-\e_k\sqrt{\chi_k}},$$
which is lower triangular. The product then follows easily by induction, and using the fact that $\prod_{k=1}^n\chi_k=1$.
\end{proof}
\begin{Rem}In particular, we see that the puncture is Ramond if $*_0>0$, i.e. the length $n$ of the loop $\c$, and the number of segments of $\c$ matching orientations with $\w$, have the same parity.
\end{Rem}
Therefore, the monodromy constraints for the Ramond punctures amount to $*_1=0$, and we arrive at the main result of the paper:
\begin{Thm}
The monodromy constraints for the Ramond punctures are given by the following linear equation of $\h_i$ around each Ramond puncture of the surface:
\Eq{
&\sum_{i=1}^{n} \h_i\prod_{j=1}^{i-1}c_i\inv=0\nonumber\\
\iff & \h_1-\h_2\frac{\e_1}{\sqrt{\chi_1}}+\h_3\frac{\e_1\e_2}{\sqrt{\chi_1\chi_2}}-.... +(-1)^{n-1}\h_n\frac{\e_1...\e_{n-1}}{\sqrt{\chi_1...\chi_{n-1}}}= 0,
}
or upon multiplication by $*_0$:
\Eq{
&\sum_{i=1}^{n} \h_i\prod_{j=i}^{n}c_i=0\nonumber\\
\iff & \h_1\sqrt{\chi_1...\chi_n}\e_1...\e_n - \h_2\sqrt{\chi_2...\chi_n}\e_2...\e_n +... +(-1)^{n-1}\h_n\sqrt{\chi_n}\e_n=0.
}
\end{Thm}
These constraints will remove the necessary $(0|n_{R})$-dimensional bundle from $S\cT(F)$, which we will call the \emph{Ramond decoration}.

\end{document}